\documentclass{article}
\usepackage{amssymb}
\usepackage[cp1251]{inputenc}
\usepackage[english]{babel}
\usepackage{amsmath}
\usepackage{amsthm}
\usepackage{amscd}
\usepackage{graphics}
\usepackage{epsfig}
\usepackage[all,knot]{xy}
\xyoption{arc}
\usepackage{color}

\newtheorem{theorem}{Theorem}[section]
\newtheorem{lemma}[theorem]{Lemma}
\newtheorem{proposition}[theorem]{Proposition}
\newtheorem{notation}[theorem]{Notation}

\newtheorem{definition}[theorem]{Definition}
\newtheorem{remark}[theorem]{Remark}
\newtheorem{example}[theorem]{Example}

\def\rg{\frak g}
\def\beq#1{\begin{equation}\label{#1}}
\def\eeq{\end{equation}}

\def\Aut{{\hbox{\bf Aut}\;}}
\def\Out{\hbox{\bf Out}\;}
\def\ad{\hbox{\bf ad}\;}
\def\Int{\hbox{\bf Int}\;}
\def\Der{\hbox{\bf Der}\;}

\def\kernel{{\hbox{\bf Ker}\;}}

\def\D{\mathcal{D}}

\def\colorf{\color{black}}

\def\beq#1{\colorf\begin{equation}\label{#1}}
\def\eeq{\end{equation}\color{black}}

\title{The Existence of Coupling in the Category of Transitive Lie Algebroid}
\author{Xiaoyu Li \\(Harbin Institute of technology), \\A.S. Mishchenko \\(Harbin Institute of technology, \\Moscow State Lomonosov University), }

\begin{document}
\maketitle

\begin{abstract}
The coupling of the tangent bundle $TM$ with the Lie algebra bundle $L$
(\cite{Mck-2005}, Definition 7.2.2) plays the crucial role
in the classification of the transitive Lie algebroids for Lie algebra  bundle $L$
with fixed finite dimensional Lie algebra $\rg$ as a fiber of $L$.
Here we give a necessary and sufficient condition for the existence of such a coupling.
Namely we define a new topology on the group $\Aut(\rg)$ of all automorphisms of Lie algebra $\rg$ and show that tangent bundle $TM$ can be coupled with the Lie algebra bundle $L$  if and only if the Lie algebra bundle L admits a local trivial structure with structural group endowed with such new topology.
\end{abstract}

\section{Introduction and Preliminaries}

For every transitive Lie algebroid, the adjoint bundle is a Lie algebra bundle and can be coupled with the tangent bundle of its base smooth manifold (See \cite{Mck-2005}, Theorem 6.5.1 and Definition 7.3.4). Conversely, there is a problem  whether there exists a transitive Lie algebroid with a given Lie algebra bundle as adjoint bundle. Mackenzie was concerned with this problem and formulated the definition of coupling and constructed the obstruction class that depends on the coupling in order to give a criterion of existence of a transition Lie algebroid(See \cite{Mck-2005}, Section 7.2). The results from  this problem are used for the description of the classifying space of
transitive Lie algebroids (See \cite{Mi-2011}). Since in \cite{Mck-2005} the coupling was assumed to  exist, it is natural to consider  the problem of existence of a coupling. In this paper, we show a necessary and sufficient condition for  existence of a  coupling of the tangent bundle $TM$ with the Lie algebra bundle $L$.

In the beginning, we shall give  some definitions and results about the Lie algebra and differential geometry.

\begin{definition}(See \cite{Hall-2004} )
Let $\rg$ be a finite dimensional Lie algebra. Let $\Aut\rg$ denote the group of Lie algebra automorphisms of $\rg$ and $\Der\rg$ denote the Lie algebra of derivations of $\rg$. Define the exponential map
$
exp:\Der\rg\rightarrow \Aut\rg
$
by the formula
$$
exp:\psi\mapsto \sum\limits_{i=1}^{\infty}\frac{\psi^{i}}{i!}
$$
where $\psi\in\Der\rg$. Usually we denote $exp(\psi)$ by $e^{\psi}$.
\end{definition}

\begin{definition}
Let $\rg$ be a finite dimensional Lie algebra. An automorphism of the form $exp(\ad u)$, where $u\in\rg$, is called \emph{inner}. More generally, the subgroup of $\Aut\rg$ generated by those is denoted by $\Int \rg$ and its elements are called \emph{inner automorphisms}.
\end{definition}

\begin{proposition}(See \cite{Mck-2005}, \cite{Hall-2004})\label{ProInt}
The subgroup $\Int \rg$ is a normal Lie subgroup of $\Aut\rg$. The Lie algebra of
the group $\Int\rg$ is $\ad\rg$.
\end{proposition}

As the group $\Int \rg$ is a normal Lie subgroup of $\Aut\rg$, there is a quotient group denoted by $\Aut\rg/\Int\rg$. Denote by $q:\Aut\rg\rightarrow \Aut\rg/\Int\rg$ the corresponding quotient map. The Lie group structure of $\Aut\rg$ is induced from the Lie group structure of the group $GL(\rg)$ (See \cite{Hatcher-2005}), where $GL(\rg)$ is the group of all linear isomorphisms from $\rg$ to $\rg$. Due to \cite{Kelly-1955}, the quotient group $\Aut\rg/\Int\rg$ has the topology induced from $\Aut\rg$ by the quotient map $q$. It is well know that  the topology of $\Aut\rg/\Int\rg$ is not always discrete. We can add more open subsets on $\Aut\rg$ such that toplogy on $\Aut\rg/\Int \rg$ becomes the discrete topology. Let us denote by $\Aut\rg^{\delta}$ the space $\Aut\rg$ with a finer topology such that the topology of $\Aut\rg/\Int \rg$ becomes the discrete topology. In order to avoid confusion, we denote by ${\Aut\rg/\Int \rg}^{d}$ the space $\Aut\rg/\Int \rg$ with discrete topology.

Let M be a smooth manifold and $\varphi:M\rightarrow \Aut\rg$ be a smooth map. Let $q:\Aut\rg\rightarrow \Aut\rg/\Int\rg$ be the quotient map defined above.
\begin{theorem}\label{LocalCon}
The composition $q\circ\varphi:M\rightarrow \Aut\rg/\Int\rg$ is locally constant if and only if $\frac{\partial\varphi}{\partial X} {\varphi(x)}^{-1}\in \ad\rg$ for arbitrary $x\in M$ and $X\in T_{x}M$.
\end{theorem}
\begin{proof}
Fix $x\in M$ and $X\in T_{x}M$. Let $\gamma:(-\varepsilon,\varepsilon)\rightarrow M$ be a curve such that $\gamma(0)=x$ and $\dot{\gamma}(0)=X$. Since $q\circ \varphi:M\rightarrow \Aut\rg/\Int\rg$ is locally constant, the map $\varphi\circ\gamma :(-\varepsilon,\varepsilon)\rightarrow \Aut\rg$ has the form
$$
\varphi\circ\gamma(t)=C\cdot\overline{\varphi}(\gamma(t)),
$$
where $C\in\Aut\rg$ is constant and $\overline{\varphi}\circ\gamma:(-\varepsilon,\varepsilon)\rightarrow \Int\rg$.
Then by Proposition \ref{ProInt}, one has
$$
\begin{array}{lll}
\frac{d\varphi\circ\gamma}{dt}|_{t=0}{\varphi(\gamma(0))}^{-1}=C\cdot\frac{d\overline{\varphi}\circ\gamma}{dt}|_{t=0}\cdot{\overline{\varphi}(\gamma(0))}^{-1}\cdot C^{-1}\in ad\rg.
\end{array}
$$
Thus $\frac{\partial\varphi}{\partial X} {\varphi(x)}^{-1}\in ad\rg$ .
\\\\
Consider the composition
$$
\xymatrix{
T_{x}M\ar[r]^{T_{*}\varphi} & T_{\varphi(x)}(\Aut\rg)\ar[r]^{T_{*}R_{\varphi(x)^{-1}}} &
T_{e}(\Aut\rg)\ar[r]^{T_{*}q}&T_{[e]}(\Aut\rg/\Int\rg)
}
$$
where $x\in M$ and $R_{\varphi(x)^{-1}}:\Aut\rg\rightarrow \Aut\rg,R_{\varphi(x)^{-1}}(\theta)=\varphi(x)^{-1}\theta$, for $\theta\in\Aut\rg$.

 We identify  $T_{e}(\Aut\rg)$ with $\Der\rg$ and $T_{[e]}(\Aut\rg/\Int\rg)$ with $\Der\rg/ad\rg$.
 For arbitrary $X\in T_{x}M$,
  $$\frac{\partial\varphi}{\partial X} {\varphi(x)}^{-1}\in ad\rg$$ is equivalent to $$T_{*}R_{\varphi(x)^{-1}}\circ T_{*}\varphi (X)\in ad\rg.$$
Then $T_{*}q\circ T_{*}R_{\varphi(x)^{-1}}\circ T_{*}\varphi(X)=0$.

Since $q\circ R_{\varphi(x)^{-1}}\circ  \varphi=R_{(q\circ \varphi(x))^{-1}}\circ (q\circ \varphi)$, it follows that  $T_{*}q\circ T_{*}\varphi(X)\equiv0$.
Consequently, $q\circ \varphi:M\rightarrow \Aut\rg/\Int\rg$ is locally constant.
\end{proof}

\begin{remark}In the theorem above, the condition that $q\circ \varphi: M\rightarrow \Aut\rg/\Int\rg$ is locally constant is equivalent to the condition that $\varphi:M\rightarrow {\Aut\rg}^{\delta}$ is continuous.
\end{remark}

It is well know that given a vector bundle endowed with a connection,  parallel transport along paths can be defined.
\begin{definition}(See \cite{Ballmann} Definition 3.1.1 and Lemma 3.1.3)\label{Parallelmap}
Let $E$ be a vector bundle on a smooth manifold $M$ and $\nabla$  be a connection on $E$. Let $\gamma:[0,1]\rightarrow M$ be a smooth path. Then for each $e\in E_{\gamma(0)}$, there exists a unique section $\sigma$ along $\gamma$ which satisfies $\nabla_{\dot{\gamma}}\sigma\equiv 0$ and $\sigma(0)=e$. Define $P_{\gamma}(e)=\sigma(1)$. Then $P_{\gamma}:E_{\gamma(0)}\rightarrow E_{\gamma(1)}$ is a well defined linear map called the \emph{parallel transport map}.
\end{definition}

\begin{lemma}(See \cite{Ballmann}, Lemma 3.1.8)\label{indep}
The parallel translation along path $\gamma$ does not depend on the parameterization of $\gamma$.
\end{lemma}

More generally, Definition \ref{Parallelmap} and Lemma \ref{indep}  extend to the piecewise smooth situation (See \cite{Ballmann}, page 22). Let $\gamma,\gamma':[0,1]\rightarrow M$ be two piecewise paths with $\gamma(1)=\gamma'(0)$. Define the  inverse path $\gamma^{-1}$ and composition $\gamma'\gamma$ by

\begin{displaymath}
\gamma^{-1}(t)=\gamma(1-t)~~~0 \leq t\leq 1~~~~;~~~~\gamma'\gamma = \left\{\begin{array}{ll}
\gamma(2t)& \textrm{$0\leq t\leq \frac{1}{2}$},\\
\gamma'(2t-1) &\textrm{$ \frac{1}{2}\leq t\leq 1$}.
\end{array}\right.
\end{displaymath}

Since the composition of smooth paths is  piecewise smooth, we have
\begin{lemma}(See \cite{Ballmann}, 3.1.18)\label{compositionpath}
Given two piecewise smooth paths $\gamma,\gamma':[0,1]\rightarrow M$ with $\gamma(1)=\gamma'(0)$. Then
$$
P_{\gamma^{-1}}={P_{\gamma}}^{-1}\quad  and \quad P_{\gamma'\gamma}=P_{\gamma'}\circ P_{\gamma}.
$$
\end{lemma}
\begin{definition}(See \cite{Ballmann}, page 23)
Let $M$ be a smooth map. Consider a continuous map
$$
H:[0,1]\times [0,1]\rightarrow M, \quad \quad \quad h_{s}=H(s,\cdot):[0,1]\rightarrow M
$$
where $h_{s}$ is a family of piecewise smooth maps, and $H(s,t)$ is smooth in $s$. The map $H$ is called  a \emph{piecewise smooth homotopy.}
\end{definition}

Consider a vector bundle $E\rightarrow M$ endowed with a connection $\nabla$ with curvature $R$. Through piecewise smooth homotopy,  we can give some important relation between parallel transport and curvature. We denoted by $P_{s,t}$ the parallel transport from $E_{h_{s}(t)}$ to $E_{h_{s}(1)}$ through path $h_{s}$ and set
$$
R_{s,t }=P_{s,t}\circ R(\partial_{t}H(s,t),\partial_{s}H(s,t))\circ P_{s,t}^{-1}:E_{h_{s}(1)}\rightarrow E_{h_{s}(1)}
$$

\begin{lemma}(See \cite{Ballmann}, Lemma 3.1.11)\label{baseofholonomy}
Let $\sigma$ be a piecewise smooth section along a piecewise smooth homotopy $H$, defined as above, such that $\nabla_{\partial _{t}}\sigma(s,t)\equiv 0$ and $\nabla_{\partial _{s}}\sigma(s,0)\equiv 0$ for all $s$. Then
$$
\nabla_{\partial_{s}}\sigma(s,1)=(\int_{0}^{1}R_{s,t}dt)\sigma(s,1).
$$

In the case $H(s,0)$ and $H(s,1)$ are constant, then $P_{s,0}=P_{h_{s}}$ and satisfies
$$
\partial_{s}P_{s,0}=(\int_{0}^{1}R_{s,t}dt)\cdot P_{s,0}.
$$
\end{lemma}

\def\e{\emph}
\section{The Existence of Coupling}

\begin{definition}(See \cite{Mck-2005},Definition 3.1.1) A \e{Lie algebroid} $A$ over a smooth manifold $M$ is a vector bundle $p:A\rightarrow M$ together with a Lie algebra structure $[\ ,\ ]$ on the space  $\Gamma(A;M)$ of sections and  a bundle map $a:A\rightarrow TM $ called the
\e{anchor}  such that
\begin{enumerate}\renewcommand{\labelenumi}{$($\roman{enumi}$)$}
\item the induced map $a:\Gamma(A;M)\rightarrow \Gamma(TM;M)$ is a Lie algebra homomorphism
\item for any sections $\sigma, \tau\in\Gamma(A;M)$ and smooth function $f\in C^{\infty}(M)$ we have the Leibniz identity $$[\sigma,f\tau]=f[\sigma,\tau]+a(\sigma)(f)\tau.$$
\end{enumerate}

\end{definition}

We call $A$ a \emph{transitive Lie algebroid} if $a$ is surjective. We often use
$$
 0\xrightarrow{} L \xrightarrow{j} A \xrightarrow{a}TM \xrightarrow{} 0
$$
 to denote a transitive Lie algebroid. Here $L=\kernel a$ is called the adjoint bundle. Sometimes we use $(A,M,[\ ,\ ],a)$ to denote a Lie algebroid in order to highlight the bracket.

 Let $(A',M,[\ ,\ ],a')$ be a second Lie algebroid on the same base $M$. Then a morphism of Lie algebroids $g:A\rightarrow A'$ over $M$, or a Lie algebroid homorphism, is a vector bundle morphism such that $a'\circ g=a$ and $f([\sigma,\tau])=[f(\sigma),f(\tau)]$, for all $\sigma,\tau\in\Gamma(A;M)$.

\begin{notation}
Given a Lie algebra bundle $L$, the transitive Lie algebroid of covariant derivations on $\Gamma(L;M)$ (See \cite{Mck-2005}, Corollary 3.6.11) is denoted by
$\D_{Der}(L)$. The bundle $\D_{Der}(L)$ is included in the Atiah exact sequence
$$
0\rightarrow \Der(L)\rightarrow  \D_{Der}(L)\rightarrow TM\rightarrow 0
$$
where $\Der(L)$ is the bundle of fiberwise derivative of $L$. The elements of $\Gamma(\D_{Der}(L);M)$ are called bracket derivations on $L$.

We denote  by $\ad L$ the image of the adjoint representation $ad:L\rightarrow \D_{Der}(L)$. One can prove that $\ad L$ is an ideal of $\D_{Der}(L)$(See \cite{Mck-2005}), page 271).
\end{notation}
\begin{definition}(See \cite{Mck-2005}, Definition 7.2.1)\label{outerderivation}
Let $L$ be a Lie algebra bundle on a  smooth manifold $M$. Then the quotient Lie algebroid
$$
0\rightarrow \Der(L)/adL\rightarrow  \D_{Der}(L)/adL\rightarrow TM\rightarrow 0,
$$
is denoted by
$$
0\rightarrow \Out_{Der}(L)\rightarrow  \Out\D_{Der}(L)\rightarrow TM\rightarrow 0
$$
and elements of $\Gamma(\Out\D_{Der}L;M)$ are called outer bracket derivations on $L$. Two quotient maps are denoted by $\natural^{+}:\Der(L)\rightarrow \Out_{Der}(L)$ and $\natural:\D_{Der}(L)\rightarrow \Out\D_{Der}(L)$.

\end{definition}

Since the tangent bundle of a smooth manifold is a Lie algebroid (\cite{Mck-2005}, Section 3.3), we have following definition.
\begin{definition}(See \cite{Mck-2005}, Definition 7.2.2)\label{akinofcoupl}
Let $L$ be a Lie algebra bundle on a smooth manifold $M$. A \e{coupling} of $TM$ with $L$ is a Lie algebroid homomorphism $\Xi:TM\rightarrow \Out\D_{Der}(L)$. We also say that $TM$  and $L$ are \e{coupled by $\Xi$}.
\end{definition}

\begin{remark}
The definition of coupling in \cite{Mck-2005} is more general, as Mackenzie defined coupling between an arbitrary Lie algebroid and Lie algebra bundle on the same base. Here we only consider the coupling between a tangent bundle and a Lie algebra bundle. Nevertheless, it is still worth considering the special  coupling in Definition \ref{akinofcoupl},  since it is useful in describing the classifying space of a  transitive Lie algebroid in \cite{Mi-2011}.
\end{remark}

Given a Lie algebra bundle, the coupling need not always exist. An example follows.
\begin{example}
Let $\mathbf{S}^{2}$ be two dimensional sphere. The tangent bundle of $\mathbf{S}^{2}$ is denoted by $T\mathbf{S}^{2}\rightarrow \mathbf{S}^{2}$ and is not trivial(See \cite{Hatcher-2005}, Page 5). Then there is no coupling $\Xi:T\mathbf{S}^{2}\rightarrow \Out\D_{Der}(T\mathbf{S}^{2})$.

To see this, we first note that $\ad (T\mathbf{S}^{2})=0$ since $T\mathbf{S}^{2}$ is an abelian Lie algebra bundle. Then $\Out\D_{Der}(T\mathbf{S}^{2})\equiv\D_{Der}(T\mathbf{S}^{2})$. Thus a coupling $\Xi$ is also a Lie algebroid homomorphism  $\Xi:T\mathbf{S}^{2}\rightarrow \D_{Der}(T\mathbf{S}^{2})$. It can be regard as  a flat connection on $T\mathbf{S}^{2}$(See \cite{Mck-2005}, Page 186).  But $T\mathbf{S}^{2}$ admits no flat connection, since $\mathbf{S}^{2}$ is simply connected, a flat connection would trivialize the bundle.
\end{example}

 In the remainder, we will be concerned with the existence of a coupling.

\begin{lemma}(See \cite{Mck-2005}, Section 7.2)\label{Lieconnection}
Let $L$ be a Lie algebra bundle on a smooth manifold $M$ with coupling $\Xi:TM\rightarrow \Out\D_{Der}(L)$. There is a connection $\nabla$ on $L$ such that
\begin{description}
  \item[i.] $\nabla_{X}[u,v]=[\nabla_{X}u,v]+[u,\nabla_{X}v]$;
  \item[ii.] $R^{\nabla}(X,Y)=\nabla_{[X,Y]}-\{\nabla_{X},\nabla_{Y}\}=ad\circ\Omega(X,Y).$
\end{description}
Here $u,v\in\Gamma(L;M)$, $X,Y\in\Gamma(TM;M)$ and $\Omega:TM\bigwedge TM\rightarrow L$ is a vector bundle morphism.
The connection $\nabla$ satisfying   property ($\mathbf{i}$) is called a \e{Lie connection}.
\end{lemma}

\begin{lemma}\label{Lieparalled}
Let $\gamma:[0,1]\rightarrow M$ be a piecewise smooth path on $M$ and $L\rightarrow M$  a Lie algebra bundle with Lie connection $\nabla$.  Then the parallel transport map $P_{\gamma}:L_{\gamma(0)}\rightarrow L_{\gamma(1)}$, as defined in Definition\ref{Parallelmap}, is a Lie algebra isomorphism.
\end{lemma}
\begin{proof}

Without lose of generality, we suppose that $\gamma:[0,1]\rightarrow M$ is smooth path. Let $u,v\in L_{\gamma(0)}$. There are sections $\sigma,\tau$ along path $\gamma$ such that $\nabla_{\dot{\gamma}}\sigma=0$, $\nabla_{\dot{\gamma}}\tau=0$ and $\sigma(\gamma(0))=u$, $\tau(\gamma(0))=v$.

By Definition \ref{Parallelmap},  $P_{\gamma}(u)=\sigma(\gamma(1))$ and $P_{\gamma}(v)=\tau(\gamma(1))$. Since $\nabla$ is a Lie connection, $\nabla_{\dot{\gamma}}([\sigma,\tau])=[\nabla_{\dot{\gamma}}\sigma,\tau]+[\sigma,\nabla_{\dot{\gamma}}\tau]=0$. Thus $[\sigma,\tau]$ is also a parallel section along $\gamma$. As $[\sigma,\tau](\gamma(0))=[u,v],$
$$
\begin{array}{lll}
P_{\gamma}([u,v])=[\sigma,\tau](\gamma(1))\\\\
~~~~~~~~~~~~~=[\sigma(\gamma(1),\tau(\gamma(1))]\\\\
~~~~~~~~~~~~~=[P_{\gamma}(u),P_{\gamma}(v)].
\end{array}
$$
Due to the above formula and Lemma \ref{compositionpath},  $P_{\gamma}$ is a Lie algebra isomorphism.
\end{proof}

\begin{theorem}\label{necesuff}
Let $L$ be a Lie algebra bundle on asmooth manifold $M$ with fibre $\rg$. Then  a coupling $\Xi:TM\rightarrow \Out\D_{Der}(L)$ exists if and only if $L$ admits a locally trivial structure with structural group ${\Aut\rg}^{\delta}$.
\end{theorem}
\begin{proof}
By Lemma\ref{Lieconnection}, there is a connection $\nabla$ on $L$ with curvature $R^{\nabla}=ad\circ \Omega$, where $\Omega:TM \bigwedge TM\rightarrow L$. Let $\{U_{\alpha},f_{\alpha}\xrightarrow{\approx} R^{n}\}_{\alpha\in\Delta}$ be a class of charts of $M$.

For each $\alpha\in\Delta$, define $x_{\alpha}\in U_{\alpha}$ by $x_{\alpha}=f_{\alpha}^{-1}(0)$. Let $x\in U_{\alpha}$ be arbitrary and define a smooth path $$\gamma_{\alpha,x}:[0,1]\rightarrow U_{\alpha}\quad by \quad \gamma_{\alpha,x}(t)=f_{\alpha}^{-1}(tf_{\alpha}(x)).$$

By Lemma \ref{Lieparalled} and the path defined above, we can define a new class of charts for $L\rightarrow M $,
\begin{equation}\label{newcharts}
\varphi_{\alpha}:U_{\alpha}\times L_{x_{\alpha}}\rightarrow L_{U_{\alpha}}, \quad\quad\varphi_{\alpha}(x,u)=P_{\gamma_{x,\alpha}}(u)
\end{equation}
for $x\in U_{\alpha}, u\in L_{x_{\alpha}}$.

If $U_{\alpha}\bigcap U_{\beta} \neq \emptyset$, then $\varphi_{\beta}^{-1}\circ\varphi_{\alpha}:U_{\alpha}\bigcap U_{\beta}\times L_{x_{\alpha}}\rightarrow U_{\alpha}\bigcap U_{\beta}\times L_{x_{\beta}}$ is
\begin{equation}\label{intersection}
\begin{array}{lll}
\varphi_{\beta}^{-1}\circ \varphi_{\alpha}(x,u)=(x,P_{\gamma_{x,\beta}^{-1}}\circ P_{\gamma_{x,\alpha}}(u))\\~~~~~~~~~~~~~~~~~~~~=(x,P_{\gamma_{x,\beta}^{-1}\gamma_{x,\alpha}}(u))
\end{array}
\end{equation}
Let us define
\begin{equation}\label{pathinters}
\gamma_{x}=\gamma_{x,\beta}^{-1}\gamma_{x,\alpha}.
\end{equation}

By (\ref{intersection}), (\ref{pathinters}) and the definition of new charts (\ref{newcharts}), the transition function $\varphi_{\alpha\beta}:U_{\alpha}\bigcap U_{\beta}\rightarrow \Aut\rg $ for the new charts is
\begin{equation}\label{transfunc}
\varphi_{\alpha\beta}(x)=P_{\gamma_{x}}.
\end{equation}

Given arbitrary $x_{0}\in U_{\alpha}\bigcap U_{\beta}$ and $X\in T_{x_{0}}U_{\alpha}\bigcap U_{\beta}$. Let $c:[-\varepsilon,\varepsilon]\rightarrow U_{\alpha}\bigcap U_{\beta}$ be a smooth path with $c(0)=x_{0}$ and $\dot{c}(0)=X$.

Define
$$H:[-\varepsilon,\varepsilon]\times [0,1]\rightarrow U_{\alpha}\bigcup U_{\beta}$$
by
\begin{displaymath}
H(s,t) = \left\{\begin{array}{ll}
f_{\alpha}^{-1}(2tf_{\alpha}(c(s)))& \textrm{$0\leq t\leq \frac{1}{2}$},\\
f_{\beta}^{-1}((2-2t)f_{\beta}(c(s))) &\textrm{$ \frac{1}{2}\leq t\leq 1$}.
\end{array}\right.
\end{displaymath}
The map $H$ is a piecewise homotopy and $H(s,\cdot)=h_{s}=\gamma_{c(s)}$, as defined in (\ref{pathinters}), for each $s\in [-\varepsilon,\varepsilon]$.

Let us use $P_{s,t}$ to denote parallel transport over the  path $h_{s}$ from $L_{h_{s}(t)}$ to $L_{h_{s}(1)}$. Since $R^{\nabla}=ad\circ \Omega$,
$$
\begin{array}{lll}
R_{s,t}=P_{s,t}\circ R(\partial_{t}H(s,t),\partial_{s}H(s,t))\circ P_{s,t}^{-1}\\\\
~~~~~~=P_{s,t}\circ (ad\circ \Omega(\partial_{t}H(s,t),\partial_{s}H(s,t)))\circ P_{s,t}^{-1}\\\\
~~~~~~=ad\circ P_{s,t}(\Omega(\partial_{t}H(s,t),\partial_{s}H(s,t)))
\end{array}
$$

By Lemma \ref{baseofholonomy},
$$
\begin{array}{lll}
\partial_{s}P_{s,0}=(\int_{0}^{1}R_{s,t}dt)\cdot P_{s,0}\\\\
~~~~~~~~=ad(\int_{0}^{1}P_{s,t}(\Omega(\partial_{t}H(s,t),\partial_{s}H(s,t)))dt)\cdot P_{s,0}.
\end{array}
$$
Since $P_{s,0}=P_{h_{s}}$ and $h_{s}=\gamma_{c(s)}$, formula (\ref{transfunc}) implies we have $P_{s,0}=\varphi_{\alpha\beta}(c(s))$.

Thus
$$
\begin{array}{lll}
\frac{d\varphi_{\alpha\beta}(c(s))}{ds}|_{s=0}\cdot \varphi_{\alpha\beta}^{-1}(c(0)) = \partial_{s}P_{s,0}|_{s=0}\cdot P_{0,0}^{-1}\\\\
~~~~~~~~~~~~~~~~~~=ad(\int_{0}^{1}P_{s,t}(\Omega(\partial_{t}H(s,t),\partial_{s}H(s,t)))dt)|_{s=0}\in\ad\rg.
\end{array}
$$
That is $\frac{\partial(\varphi_{\alpha\beta})}{\partial X}\cdot \varphi_{\alpha\beta}(x_{0})^{-1}\in ad\rg$.

Then by Theorem \ref{LocalCon}, we have  $\varphi_{\alpha\beta}:U_{\alpha}\bigcap U_{\beta}\rightarrow {\Aut\rg}^{\delta}$. So the Lie algebra bundle $L\rightarrow M$ admits structural group ${\Aut\rg}^{\delta}$.
\\\\
Now let us prove the other direction. As $L\rightarrow M$ admits structural group
${\Aut\rg}^{\delta},$ there is a class of charts $\{U_{\alpha},\varphi_{\alpha}:U_{\alpha}\times\rg\rightarrow L_{U_{\alpha}}\}_{\alpha\in\Delta}$ with transition functions $\{\varphi_{\alpha\beta}:U_{\alpha}\bigcap U_{\beta}\rightarrow {\Aut\rg}^{\delta}\}_{\alpha,\beta\in\Delta}$.

For each $\alpha\in\Delta$, we define a connection $\nabla^{\alpha}$ on $L_{U_{\alpha}}$ that is a bundle morphism $\nabla^{\alpha}:TU_{\alpha}\rightarrow \D_{Der}(L)|_{U_{\alpha}}$, by
\begin{equation}\label{locconn}
\nabla_{X}^{\alpha}s_{\alpha}=\varphi_{\alpha}(\frac{\partial \varphi_{\alpha}^{-1}(s_{\alpha})}{\partial X})
\end{equation}
where $X\in\Gamma(TU_{\alpha};U_{\alpha})$ and $s_{\alpha}\in\Gamma(L_{U_{\alpha}};U_{\alpha})$.

Let $\{h_{\alpha}\}_{\alpha\in\Delta}$ be a partition of unity corresponding to $\{U_{\alpha}\}_{\alpha\in\Delta}$. Then define a connection $\nabla$ on $L$, that is $\nabla:TM\rightarrow \D_{Der}(L)$, by the formula
$$
\nabla_{X}s(m)=\sum\limits_{\alpha\in \Delta}h_{\alpha}(m)\cdot \nabla_{X}^{\alpha}s|_{U_{\alpha}}(m)
$$
where $X\in\Gamma(TM;M)$, $s\in\Gamma(L;M), m\in M$ and the corresponding items in the sum are considered to be zero when $m\notin U_{\alpha}$.

Let us define $\Xi^{\alpha}:TU_{\alpha}\rightarrow \Out\D_{Der}(L)|_{U_{\alpha}}$ by
$$
\Xi^{\alpha}=\natural \circ \nabla^{\alpha}.
$$
By straightforward calculation, for  $X_{\alpha},Y_{\alpha}\in\Gamma(TU_{\alpha};U_{\alpha})$,
$$
\begin{array}{lll}
R^{\Xi^{\alpha}}(X_{\alpha},Y_{\alpha})=\Xi^{\alpha}([X_{\alpha},Y_{\alpha}])-\{\Xi^{\alpha}(X_{\alpha}),\Xi^{\alpha}(Y_{\alpha})\}~\\\\
~~~~~~~~~~~~~~~~~~=\natural\circ R^{\nabla^{\alpha}}(X_{\alpha},Y_{\alpha})=0
\end{array}
$$

Define a bundle homorphism $\Xi:TM\rightarrow \Out\D_{Der}(L)$ by
$$\Xi=\natural\circ \nabla.$$ Actually, due to the definition of $\Xi^{\alpha}$,
\begin{equation}\label{coupling}
\Xi(X_{m})=\sum\limits_{\alpha\in\Delta}h_{\alpha}(m)\cdot \Xi^{\alpha}(X_{m})
\end{equation}
for $X_{m}\in T_{m}M$ and the corresponding items in the sum are considered to be zero if $m\notin U_{\alpha}$.

In the case $U_{\alpha}\bigcap U_{\beta}\neq \emptyset$, choose arbitrary $X\in\Gamma(TU_{\alpha}\bigcap U_{\beta};U_{\alpha}\bigcap U_{\beta})$ and section $s\in\Gamma(L_{U_{\alpha}\bigcap U_{\beta}};U_{\alpha}\bigcap U_{\beta})$. Then
$$
\begin{array}{lll}
\nabla^{\beta}_{X}s=\varphi_{\beta}(\frac{\partial \varphi_{\beta}^{-1}(s)}{\partial X})\\\\
~~~~~~~=\varphi_{\beta}(\frac{\partial\varphi_{\beta}^{-1}\circ\varphi_{\alpha}\circ\varphi_{\alpha}^{-1}(s)}{\partial X})
=\varphi_{\beta}(\frac{\partial(\varphi_{\alpha\beta}\circ\varphi_{\alpha}^{-1}(s))}{\partial X})\\\\
~~~~~~~=\varphi_{\beta}\circ \varphi_{\alpha\beta}\circ (\frac{\partial \varphi_{\alpha}^{-1}(s)}{\partial X})+ \varphi_{\beta}\circ(\frac{\partial \varphi_{\alpha\beta}}{\partial X})\circ \varphi_{\alpha}^{-1}(s)\\\\
~~~~~~~=\varphi_{\beta}\circ \varphi_{\beta}^{-1}\circ \varphi_{\alpha}\circ (\frac{\partial \varphi_{\alpha}^{-1}(s)}{\partial X})+\varphi_{\beta}\circ(\frac{\partial \varphi_{\alpha\beta}}{\partial X})\circ \varphi_{\alpha}^{-1}(s)\\\\
~~~~~~~=\nabla^{\alpha}_{X}s+\varphi_{\alpha}\circ\varphi_{\alpha}^{-1}\circ \varphi_{\beta}\circ (\frac{\partial \varphi_{\alpha\beta}}{\partial X})\circ \varphi_{\alpha}^{-1}(s)\\\\
~~~~~~~=\nabla^{\alpha}_{X}s+\varphi_{\alpha}\circ (\varphi_{\alpha\beta}^{-1}\cdot (\frac{\partial \varphi_{\alpha\beta}}{\partial X}) )\circ \varphi_{\alpha}^{-1}(s)\\\\
~~~~~~~=\nabla^{\alpha}_{X}s+[\varphi_{\alpha}(h_{\alpha\beta}(X)),s]
\end{array}
$$
here $h_{\alpha\beta}:TU_{\alpha}\bigcap U_{\beta}\rightarrow U_{\alpha}\bigcap U_{\beta}\times\rg $ such that $\varphi_{\alpha\beta}^{-1}\cdot (\frac{\partial \varphi_{\alpha\beta}}{\partial X})=ad\circ h_{\alpha\beta}(X)$ is from the fact that $\varphi_{\alpha\beta}:U_{\alpha}\bigcap U_{\beta}\rightarrow {\Aut\rg}^{\delta}$ and Theorem \ref{LocalCon}.

Thus $\nabla_{X}^{\alpha}-\nabla_{X}^{\beta}\in adL$. So $\Xi^{\alpha}=\Xi^{\beta}$ on $TU_{\alpha}\bigcap U_{\beta}$.
Hence $\Xi|_{TU_{\alpha}}=\Xi^{\alpha}$. The bundle morphism $\Xi:TM\rightarrow \Out\D_{Der}(L)$ defined in (\ref{coupling}) is a coupling since $R^{\Xi}|_{U_{\alpha}}=R^{\Xi^{\alpha}}=0$ for arbitrary $\alpha\in\Delta$.

\end{proof}

\section*{Acknowledgement}
The second version of the article is different from the first because of the improvement of the text by invaluable efforts of professor J.Stasheff, to whom we submit our sincere gratitude.

\bibliographystyle{plain}
\bibliography{Mish-li-2013-10-22-2nd}
\end{document}